\documentclass[11pt,letterpaper]{amsart}

\usepackage{amsmath,amssymb,amsthm,amscd}
\usepackage{mathrsfs}
\usepackage[shortlabels]{enumitem}
\setlist[enumerate]{font=\normalfont}
\usepackage{mathtools}
\usepackage{color}
\usepackage{bbm}
\usepackage{tikz-cd}
\usepackage{comment}

\usepackage[hmargin=2.5cm,vmargin=2cm]{geometry}
\usepackage[normalem]{ulem}
\usepackage{tcolorbox}
\usepackage{mathabx}

\usepackage[draft=false]{hyperref}
\usepackage[noabbrev, capitalise, nameinlink]{cleveref}                                      

\usepackage{float}
\hypersetup{
colorlinks=true,
linkcolor={black},
citecolor={black}
}

\newtheorem{theorem}{Theorem}[section]
\newtheorem*{theorem*}{Theorem}
\newtheorem{lemma}[theorem]{Lemma}
\newtheorem*{lemma*}{Lemma}
\newtheorem{corollary}[theorem]{Corollary}

\newtheorem{proposition}[theorem]{Proposition}
\newtheorem{question}[theorem]{Question}
\newtheorem{example}[theorem]{Example}

\theoremstyle{definition} 
\newtheorem{definition}[theorem]{Definition} 
\newtheorem{remark}[theorem]{Remark} 

\newtheorem*{notation*}{Notation}

\numberwithin{equation}{section}

\newcommand{\cP}{\mathcal P}

\def\N{\mathbb{N}}

\def\Z{\mathbb{Z}}

\DeclareMathOperator{\sign}{sign}

\newcommand{\BF}{\textup{BF}}

\subjclass[2020]{Primary 37B10, Secondary 37A55}

\keywords{Unital shift equivalence, eventual conjugacy, continuous orbit equivalence, C*-algebras}

\title{Unital shift equivalence}
\author[K. A. Brix]{Kevin Aguyar Brix}
\address[K. A. Brix]{Department of Mathematics and Computer Science, University of Southern Denmark, 5230 Odense, Denmark}
\email{kabrix@imada.sdu.dk}

\author[E. Ruiz]{Efren Ruiz}
\address[E. Ruiz]{Department of Mathematics, University of Hawaii, Hilo, 200 W. Kawili St., Hilo, Hawaii, 96720-4091 U.S.A.} 
\email{ruize@hawaii.edu}

\thanks{KAB was supported by a Reintegration Fellowship from the Carlsberg Foundation (CF23--1328).}

\begin{document}

\setlength{\parindent}{0cm} 
\setlength{\parskip}{0.5cm}

\maketitle
\begin{abstract}
We introduce and study a unital version of shift equivalence for finite square matrices over the nonnegative integers.
In contrast to the classical case, we show that unital shift equivalence does not coincide with one-sided eventual conjugacy.
We also prove that unital shift equivalent matrices define one-sided shifts of finite type that are continuously orbit equivalent. 
Consequently, unitally shift equivalent matrices have isomorphic topological full groups and isomorphic Leavitt path algebras, the latter being related to Hazrat's graded classification conjecture in algebra.
\end{abstract}

\section{Introduction}
Shift equivalence has played a prominent role in the still open problem of classifying shifts of finite type up to conjugacy, dating back to Williams' seminal work \cite{Williams1973}. 
Although shift equivalence does not provide a complete description of conjugacy \cite{Kim-Roush92,Kim-Roush99}, it is an important relation for shifts of finite type that is decidable and has spawned an immense body of work, see e.g. \cite[Chapter 7]{Lind-Marcus}.
It also provides deep connections to other mathematical branches such as algebra \cite{Hazrat2013, Abrams-Ruiz-Tomforde} and C*-algebra theory \cite{Bratteli-Kishimoto, Carlsen-DorOn-Eilers}. 
Two square matrices $A$ and $B$ with nonnegative integer entries are \emph{shift equivalent} if there are nonnegative integer rectangular matrices $R$ and $S$ and an integer $\ell\geq 1$ satisfying the equations
\[
  A^\ell = RS,\quad B^\ell = SR, \quad AR=RB, \quad BS=SA.
\]
The matrices $A$ and $B$ determine stationary inductive systems whose limits are dimension groups (we give details below) with a canonical order structure and $\Z[x,x^{-1}]$ action,
and the matrices $R$ and $S$ induce isomorphisms of these dimension modules attached to $A$ and $B$, respectively.
It transpires that the dimension module consisting of this dimension group together with its order structure and a canonical automorphism is a complete description of shift equivalence. 

The purpose of this short note is three-fold:
\begin{enumerate}
    \item First, we wish to call attention to a unital version of shift equivalence that is relevant for the study of one-sided shifts of finite type.
    \item Second, we show by examples that unital shift equivalence does not coincide with one-sided eventual conjugacy (in contrast to the two-sided case) nor with Matsumoto's eventual conjugacy; consequently, we do not yet have a dynamical description of unital shift equivalence.
    \item Third, we prove that unital shift equivalence implies continuous orbit equivalence for all shifts of finite type, and this should be viewed as a one-sided version of Boyle's recent result that shift equivalence implies flow equivalence \cite{Boyle2024}.
\end{enumerate}

For a shift equivalence to be \emph{unital}, we require that naturally defined order units in the dimension groups are preserved under the isomorphisms induced by $R$ and $S$, see \cref{def:SE+}.
From the point of view of \emph{one-sided} shifts of finite type, this unital version is a more natural relation to study.
There are important examples of shift equivalent matrices where we do not know if they are in fact strong shift equivalent, and we show here that many of them are in fact unitally shift equivalent. This means that an analysis of the one-sided shifts of finite type could provide additional insights into these problems. 

Symbolic dynamics enjoys fruitful connections with C*-algebra theory, see e.g. \cite{Cuntz-Krieger1980,Eilers-Restorff-Ruiz-Sorensen}, and pure algebra \cite{Abrams-Pino2005} ,
and unital shift equivalence fits directly into one of the prominent conjectures put forth by Hazrat \cite[Conjecture 1]{Hazrat2013}:
two matrices are unitally shift equivalent if and only if their Leavitt path algebras are graded isomorphic.
See \cite{Abrams-Ruiz-Tomforde, Brix-DorOn-Hazrat-Ruiz} for recent advances on this problem.
There is an analogous C*-algebraic conjecture: two matrices are unitally shift equivalent if and only if their Cuntz--Krieger C*-algebras are *-isomorphic in a gauge-equivariant way.
A consequence for C*-algebras of our result that unital shift equivalence implies continuous orbit equivalence is that a gauge-equivariant *-isomorphism between Cuntz--Krieger algebras implies the existence of a (possibly different) diagonal-preserving *-isomorphism.
We do not know of a C*-algebraic proof of such a result.

We know from work of Williams and of Kim and Roush that matrices are shift equivalent if and only if the two-sided shifts of finite type they determine are eventually conjugate in the sense that all but finitely many higher powers are conjugate \cite[Theorem 7.5.15]{Lind-Marcus}.
At the moment, we do not know of a dynamical description of unital shift equivalence.
We verify that \emph{one-sided eventual conjugacy} (i.e. having conjugate higher powers) as studied by Boyle, Fiebig, and Fiebig \cite{Boyle-Fiebig-Fiebig} is too strong, and we also show that Matsumoto's notion of one-sided eventual conjugacy is too strong.
Examples show Matsumoto's eventual conjugacy does not imply having conjugate higher powers, and this reveals a discrepancy in the terminology that we clarify.
We leave the converse problem open: \emph{does having conjugate higher powers imply Matsumoto's eventual conjugacy?}.
We emphasise that an affirmative answer to this problem will resolve the question of whether Ashley's graph is strong shift equivalent to the matrix $(2)$.
The three competing notions (unital shift equivalence, having conjugate higher powers, and Matsumoto's eventual conjugacy) are therefore mutually distinct, and we depict them in the diagram below in which none of the arrows can be inverted
{\normalsize
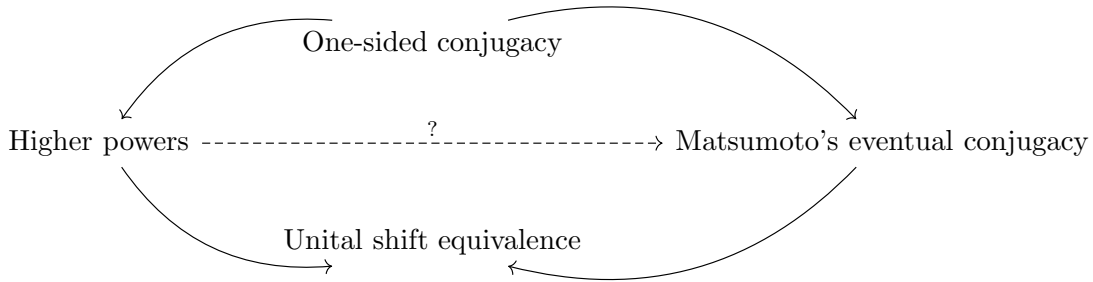
  \[
    \begin{tikzcd}
      & \textup{One-sided conjugacy} \arrow[dl, bend right] \arrow[dr, bend left] & \\
      \textup{Higher powers} \arrow[dr, bend right] \arrow[rr, dashed, "?"] &  & \textup{Matsumoto's eventual conjugacy} \arrow[dl, bend left]  \\
      & \textup{Unital shift equivalence} &
    \end{tikzcd} 
  \]
  }

Very recently, Boyle resolved the problem of whether shift equivalence implies flow equivalence in the affirmative \cite{Boyle2024}.
Seeing that continuous orbit equivalence is the one-sided analogue of flow equivalence, we were prompted to verify that our notion of unital shift equivalence indeed implies continuous orbit equivalence. We provide the details in the proof of \cref{thm:SE+->COE}.

The paper is structured as follows: 
After a preliminary section in \cref{sec:prelim},
we introduce unital shift equivalence in \cref{sec:SE+} which may be thought of as a one-sided analogue of shift equivalence.
In \cref{sec:discrepancy}, we discuss the connection between unital shift equivalence and having conjugate higher powers as well as Matsumoto's eventual conjugacy. This reveals an unfortunate discrepancy in the terminology in the literature.
In \cref{sec:SE+->COE}, we prove that unital shift equivalence implies continuous orbit equivalence. This is a one-sided analogue of Boyle's recent result that shift equivalence implies flow equivalence.

\section*{Acknowledgements}
We thank Mike Boyle for numerous useful comments on early drafts of this paper and Søren Eilers for finding the path of matrices in \cref{ex:Rourke} using a computer search.

\section{Preliminaries} \label{sec:prelim}
For a general introduction to symbolic dynamics, we refer the reader to \cite{Lind-Marcus}.
We let $\N$ denote the nonnegative integers including zero.
An $\N$-matrix $A$ is a matrix over $\N$, and if $A$ is an $n\times n$ matrix, we let $|A| = n$ denote the size of $A$.
A square $\N$-matrix $A$ is the adjacency matrix of a directed graph with $|A|$ vertices and $A(i,j)$ edges from vertex $i$ to vertex $j$.
Recall, $A$ is \emph{irreducible} if for any $1\leq i,j\leq n$, there is $N\in \N$ such that $A^N(i,j) > 0$ (equivalently, the directed graph determined by $A$ is strongly connected);
in particular, the zero matrix is not irreducible.
A \emph{sink} in a directed graph is a vertex that emits no edges (this is a zero row in the adjacency matrix),
and a \emph{source} is a vertex that receives no edges (this is a zero column in the matrix).

Let $E$ be a finite directed graph with discrete edge set $E^1$, vertex set $E^0$, and range and source maps $r$ and $s$, respectively.
Let $A$ be the adjacency matrix of $E$.
The \emph{one-sided edge shift} determined by $A$ is the compact Hausdorff space
\[
X_A \coloneq \{ (x_i)_i \in (E^1)^\N : r(x_i) = s(x_{i+1}), i\in \N \},
\]
in the topology inherited from $(E^1)^\N$, together with the shift given by $\sigma_A(x)_i = x_{i+1}$, for all $x = (x_i)_i \in X_A$.
This is a local homeomorphism.
Given two adjacency matrices $A$ and $B$, the edge shifts $(\sigma_A,X_A)$ and $(\sigma_B,X_B)$ are \emph{one-sided conjugate} if there exists a homeomorphism $h\colon X_A \to X_B$ such that $h\circ \sigma_A = \sigma_B\circ h$.

We recall outsplits and insplits for finite directed graphs with no sinks in terms of their adjacency matrices (see \cite[Section 2.4]{Lind-Marcus} for more detailed explanations).
Strong shift equivalence is then the equivalence relation generated by outsplits and insplits (and their inverses) by Williams' theorem \cite[Section 7]{Lind-Marcus}.
A rectangular $\{0,1\}$-matrix $D$ is a \emph{division matrix} if every row contains at least one $1$ and every column contains exactly one $1$.

Let $A$ and $B$ be a square $\N$-matrix with no zero rows. 
We say $B$ is an \emph{outsplit} of $A$ if there are a division matrix $D$ and an $\N$-matrix $E$ such that $A = DE$ and $B = ED$.
If this is the case, we say $A$ is an \emph{outamalgamation} of $B$.
The process of outamalgamation may be iterated at most finitely many times before no further amalgamations can occur, and the resulting matrix is called a \emph{total amalgamation}.
Every square $\N$-matrix with no zero rows admits a total amalgamation which is in fact unique up to conjugation by a permutation matrix.
Moreover, $A$ and $B$ determine edge shifts that are one-sided conjugate if and only if their total amalgamations agree (up to conjugation by a permutation matrix).

Similarly, $B$ is an \emph{insplit} of $A$ if there are a division matrix $D$ and an $\N$-matrix $E$ such that $A = ED^t$ and $B = D^t E$,
and if this is the case, we say $A$ is an \emph{inamalgamation} of $B$.
The process of insplitting does not preserve the one-sided conjugacy class of the edge shifts in general.
As in \cite{Brix2022}, we say $A$ and $B$ are \emph{balanced elementary strong shift equivalent} if there are a division matrix $D$ and rectangular $\N$-matrices $R_A$ and $R_B$ such that 
\[
A = D^t R_A, \quad B= D^t R_B,\quad R_A D^t = R_B D^t.
\]
This means that both $A$ and $B$ are insplits of a common matrix by the same (transposed) division matrix.
We let \emph{balanced strong shift equivalence} be the equivalence relation generated by balanced elementary strong shift equivalence and outsplits.

\section{Unital shift equivalence}
\label{sec:SE+}

In this section we discuss shift equivalence of matrices and introduce unital shift equivalence.
We refer the reader to \cite{Lind-Marcus} for general background on symbolic dynamics, and specifically Chapter 7 where shift equivalence is discussed.
For a recent general review of shift equivalence, we refer to \cite{Boyle-Schmieding}.

A square $\N$-matrix $A$ defines a linear map $\Z^{|A|} \to \Z^{|A|}$ where the matrix acts on column vectors.
The \emph{dimension group} of $A$ is the limit of the stationary inductive system
\[
  \Z^{|A|} \overset{A^t}{\to} \Z^{|A|} \overset{A^t}{\to} \cdots 
\]
where $A^t$ is the transpose of $A$.
The limit is isomorphic to the group 
\[
  G_A \coloneq (\Z^{|A|}\times \N)/\sim,
\]
where $(v,k)\sim (v',k')$ if and only if there exists $l\in \N$ such that $(A^t)^{l-k}v = (A^t)^{l-k'}v'$.
This group carries a natural positive cone of elements $[v,k]$ where $v\in \N^{|A|}$.
Multiplication by $A^t$ defines an automorphism $\theta_A$ on $G_A$ that preserves the positive cone.
This defines a $\Z[x,x^{-1}]$-module structure on $G_A$, and we refer to the triple $(G_A, G_A^+,\theta_A)$ as the \emph{dimension data} of $A$.
If $\underline 1 \coloneq (1,\cdots, 1)^t$ denotes the column vector consisting only of $1$'s, then $u_A \coloneq [\underline 1,0]$ defines an order unit in $G_A$.
We refer to the quadruple $(G_A, G_A^+,\theta_A, u_A)$ as the \emph{unital dimension data}.
Given $A$ and $B$, their unital dimension data are said to be isomorphic if there is a group isomorphism $\Theta\colon G_A \to G_B$ that preserves the positive cone, intertwines the canonical automorphisms $\theta_A$ and $\theta_B$,
and satisfies $\Theta[\underline 1, 0]_A = [\underline 1,0]_B$.

A pair of square $\N$-matrices $A$ and $B$ are \emph{shift equivalent} if there are $\ell\in \N$ (the lag) and rectangular $\N$-matrices $R$ and $S$ satisfying
\[
  A^\ell = RS,\quad B^\ell = SR, \quad AR=RB, \quad BS=SA.
\]
A shift equivalence $(R,S)$ defines an isomorphism of the dimension groups $\hat{R}\colon G_A \to G_B$ given by
\begin{equation} \label{eq:R-dimension-group}
  \hat{R}([v,k]) = [R^tv,k],
\end{equation}
for all $[v,k]\in G_A$,
and this preserves the positive cone and it intertwines the automorphisms $\theta_A$ and $\theta_B$.

\begin{definition} \label{def:SE+}
Let $A$ and $B$ be square $\N$-matrices.
We say a shift equivalence $(R,S)$ between $A$ and $B$ is \emph{unital} if the induced map $\hat{R}\colon G_A \to G_B$ from \eqref{eq:R-dimension-group} preserves the unit 
in the sense that $\hat{R}(u_A) = \hat{B}^k u_B$, for some $k\in \N$, equivalently, $(B^t)^mR^t \underline{1} = (B^t)^{m+k} \underline{1}$ for some $m, k \in \N$.
We say $A$ and $B$ are \emph{unitally shift equivalent}, if there exists a unital shift equivalence from $A$ to $B$.
\end{definition}

\begin{proposition} \label{prop:SE+-dimension-data}
    Let $A$ and $B$ be square $\N$-matrices.  Then, $A$ and $B$ are unitally shift equivalent if and only if their unital dimension data are isomorphic.
\end{proposition}

\begin{proof}
    The only nontrivial statement is if the unital dimension data are isomorphic, then $A$ and $B$ are unitally shift equivalent.  
    Suppose $\Theta \colon G_A \to G_B$ is an isomorphism of the unital dimension data.
    By Krieger's Theorem \cite{Kr-shift} (see also \cite[Theorem 6.4]{Effros1979}), there exist $k \in \N$ and a shift equivalence $(R,S)$ between $A$ and $B$ such that $\hat{B}^k \circ \Theta = \hat{R}$. 
    But then $\hat{R}_Au_A = \hat{B}^k u_B$, so $(R,S)$ is a unital shift equivalence.
\end{proof}

\begin{remark}
It is wellknown that for primitive matrices, shift equivalence is characterised by the dimension group and its canonical automorphism (not including the positive cone) \cite[Corollary 7.5.9]{Lind-Marcus}; similarly, unital shift equivalence is characterised by the unital dimension group with its canonical automorphism for primitive matrices.    
\end{remark}

Since one-sided conjugacy of edge shifts is generated by outsplits and outamalgamations of the adjacency matrices, the observation below shows that one-sided conjugate edge shifts have unitally shift equivalent matrices.
The converse is not the case and we give (several) examples of this.

\begin{lemma} \label{lem:one-sided-conjugacy-SE+}
    Let $A$ and $B$ be a square $\N$-matrix with no zero rows, and suppose $B$ is an outsplit of $A$.
    Then, $A$ and $B$ are unitally shift equivalent.
\end{lemma}

\begin{proof}
    If $B$ is the result of an outsplit of $A$, then there are rectangular $\{0,1\}$-matrices $D$ and $E$ satisfying $A=DE$ and $B=ED$, see e.g. \cite[Theorem 2.4.12]{Lind-Marcus}.
    The pair $(D,E)$ is a shift equivalence, and $D$ has the property that every row of $D^t$ contains exactly one $1$.
    Therefore, $D^t \underline 1 = \underline 1$ from which we see that $(D,E)$ is a unital shift equivalence.
\end{proof}

The example below shows that insplits need not preserve the unital shift equivalence class of a matrix.
In particular, strong shift equivalence does not imply unital shift equivalence.

\begin{example}
\begin{enumerate}[label=(\arabic*)]
    \item For the matrix $A=(2)$, 
    \[
    \lambda \colon (G_A, G_A^+, \theta_A, u_A)\cong \left( \Z\left[\frac{1}{2}\right], \Z\left[\frac{1}{2}\right]^+, \operatorname{mult}_2, 1\right),
    \]
    where $\operatorname{mult}_2$ is the multiplication by $2$ automorphism on the dyadic rationals $\Z[\frac{1}{2}]$.

    \item The matrix 
    \[
    B \coloneq \begin{pmatrix}
        1 & 1 \\
        1 & 1
    \end{pmatrix}
    \]
    is an outsplit of $A$ via $R = \begin{pmatrix} 1 & 1 \end{pmatrix}$ and $S = \begin{pmatrix} 1 \\ 1 \end{pmatrix}$, which implies 
    \[
    (G_B, G_B^+, \theta_B, u_B)\cong \left( \Z\left[\frac{1}{2}\right], \Z\left[\frac{1}{2}\right]^+, \operatorname{mult}_2, 2\right), \ [ v,k ] \mapsto \lambda([ S^t v , k ])
    \]
    Since $2\in \Z[\frac{1}{2}]$ is in the orbit of $\operatorname{mult}_2$, the unital dimension data of $A$ and $B$ are isomorphic.  By \cref{prop:SE+-dimension-data}, $A$ and $B$ are unitally shift equivalent.
    \item The matrix
    \[
    C \coloneq \begin{pmatrix}
        1 & 0 & 1 \\
        1 & 0 & 1 \\
        0 & 1 & 1 
    \end{pmatrix}
    \]
    in an insplit of $B$ via 
    \[
    R \coloneq \begin{pmatrix}
        1 & 0 & 1 \\
        0 & 1 & 1 
    \end{pmatrix}\qquad \textup{and} \qquad 
    S \coloneq \begin{pmatrix}
        1 & 0 \\
        1 & 0 \\
        0 & 1
    \end{pmatrix},
    \]
    which implies 
    \[
    (G_C, G_C^+, \theta_C, u_C)\cong \left( \Z\left[\frac{1}{2}\right], \Z\left[\frac{1}{2}\right]^+, \operatorname{mult}_2, 3\right), [ v, k ] \mapsto \lambda([ \begin{pmatrix} 1 & 1 \end{pmatrix} S^t v, k]).
    \]
    Note that there is no automorphism of $\Z[\frac{1}{2}]$ that sends $1$ to $3$.  Thus, the unital dimension data of $C$ is not isomorphic to the unital dimension data of $A$.  By \cref{prop:SE+-dimension-data}, $C$ is not unitally shift equivalent to $B$ (or $A$).
\end{enumerate}
\end{example}

\begin{remark}
A pointed version of Hazrat's graded classification conjecture for Leavitt path algebras \cite{Hazrat2013} predicts that the pointed graded $K_0$-group of the Leavitt path algebra of finite graphs with no sinks are isomorphic if and only if the Leavitt path algebras are graded isomorphic (see also \cref{cor:LPA-isomorphism}).
The pointed graded $K_0$-group coincides with the unital dimension data we consider here for the underlying adjacency matrices, so our notion of unital shift equivalence provides a matrix aspect of this conjecture.
In our previous work \cite{Brix-DorOn-Hazrat-Ruiz} with Dor-On and Hazrat, we discussed a slightly stronger version of unital shift equivalence (see Section 2.1 in that paper) 
where we said a shift equivalence $(R,S)$ is unital, if $\hat{R}(u_A) = u_B$.  We were able to prove that the Leavitt path algebras of finite graphs with no sinks are isomorphic provided that their adjacency matrices are unital shift equivalent (in the sense of \cite{Brix-DorOn-Hazrat-Ruiz}) and they satisfy an alignment condition (see \cite[Definitions~5.2 and 5.4]{Brix-DorOn-Hazrat-Ruiz}). We choose this slightly more flexible unital shift equivalence (\cref{def:SE+}) in order to establish that unital shift equivalence coincides with having isomorphic unital dimension data.  

There are also C*-algebraic versions of Hazrat's conjecture concerning gauge-equivariant isomorphism of graph C*-algebras, and this is e.g. studied in \cite{Eilers-Ruiz} where the unital dimension data appears as the dimension quadruple.
See also \cite{Bratteli-Kishimoto}.
\end{remark}

The \emph{unital signed Bowen--Franks group} of $A$ is the cokernel $\BF(A) \coloneq \Z^{|A|}/(I-A^t)\Z^{|A|}$ together with the class of the vector $\underline{1}\in \Z^{|A|}$ in the cokernel
and the sign of the determinant $\sign \det I - A^t$.

\begin{lemma}
    The unital Bowen--Franks group is an invariant of unital shift equivalence.
\end{lemma}

\begin{proof}
It is wellknown that the sign of the determinant is invariant under shift equivalence (this follows e.g. from \cite[Corollary 7.4.12]{Lind-Marcus}).

There is a canonical map $q_A\colon G_A \to \BF(A)$ given by $q_A([v,k]) = v + (I-A^t)\Z^{|A|}$, for all $[v,k] \in G_A$.
In particular, $q_A([\underline 1,0]) = \underline 1 + (I-A^t)\Z^{|A|}$.
A shift equivalence $(R,S)$ from $A$ to $B$ induces a group isomorphism $\check{R}\colon \BF(A) \to \BF(B)$ given by
\begin{equation} \label{eq:R-Bowen-Franks}
    \check{R}(v+ (I-A^t)\Z^{|A|}) = R^tv + (I-B^t)\Z^{|B|},
\end{equation}
for all $v\in \Z^{|A|}$ with inverse $\check{S}( w + (I-B^t)\Z^{|B|} ) = S^tw+ (I-A^t)\Z^{|A|}$ as $A^t$ and $B^t$ induce identity maps on the Bowen--Franks groups of $A$ and $B$, respectively.
An easy computation shows that 
\begin{equation} \label{eq:R-hat-check}
q_B\circ \hat{R} = \check{R}\circ q_A.
\end{equation}
In particular, if $(R,S)$ is unital, then $\check{R}$ respects the classes of $\underline 1$ in the Bowen--Franks groups.
\end{proof}

\begin{example}
For integers $k\geq 2$, the matrices below 
\[
A_k =
\begin{pmatrix}
    1 & k \\
    k-1 & 1
\end{pmatrix}\quad \textup{and} \quad
B_k = 
\begin{pmatrix}
    1 & (k-1)k \\
    1 & 1
\end{pmatrix},
\]
are known to be shift equivalent,
and it is an important open problem to determine if they are strong shift equivalent (only the cases $k=2,3$ are confirmed to be strong shift equivalent), see \cite[Example 7.3.13]{Lind-Marcus}.
The matrix $P_k = \begin{pmatrix} k-1 & k \\ 1 & 1 \end{pmatrix}$ implements a similarity over $\Z$ of $A_k$ and $B_k$,
and a concrete shift equivalence is given as $R = P_k^{-1} B_k^j$ and $S = B_k P_k A_k^j$, for some $j>1$.
Observe that $(B_k)^t (P_k^{-1})^t\underline 1 = \underline 1$, so $R^t \underline 1 = (B_k^{j-1})^t \underline 1$.
It follows that $(R,S)$ is a unital shift equivalence between $A_k$ and $B_k$.
\end{example}

We give more examples of matrices that are unitally shift equivalent in the next section.

\section{A discrepancy in one-sided eventual conjugacy}
\label{sec:discrepancy}

In this section, we discuss the connections between unital shift equivalence, Matsumoto's eventual conjugacy, and having (one-sided) conjugate higher powers. 
The latter two notions are both called \emph{eventual conjugacy} in the literature, so one main point of this section is that all three notions are in fact distinct.
This is in contrast to the two-sided case, where shift equivalence coincides with having (two-sided) conjugate higher powers.

Matsumoto \cite{Matsumoto2017} introduces \emph{eventual conjugacy} for one-sided shifts of finite type.
Two one-sided shifts of finite type $\sigma_A$ and $\sigma_B$ are \emph{Matsumoto eventually conjugate} if there exist a homeomorphism $h\colon X_A \to X_B$ and an integer $\ell\in \N$ satisfying
\begin{align}
    \sigma_B^{\ell+1}\circ h(x) = \sigma_B^\ell \circ h(\sigma_A(x)), \\
    \sigma_A^{\ell+1}\circ h^{-1}(y) = \sigma_A^\ell \circ h^{-1}(\sigma_B(y)),
\end{align}
for all $x\in X_A$ and $y\in X_B$.

In \cite{Brix2022}, it is shown that Matsumoto's eventual conjugacy is generated at the level of directed graphs by out-splits and balanced in-split,
and an algebraic characterisation called balanced strong shift equivalence of Matsumoto's eventual conjugacy is established.
It is also argued that Matsumoto's eventual conjugacy should be viewed as the one-sided analogue of two-sided conjugacy (strong shift equivalence).

\begin{proposition}\label{prop:111->SE+}
  Let $A$ and $B$ be square $\N$-matrices.
  If $A$ and $B$ are balanced strong shift equivalent, then $A$ and $B$ are unitally shift equivalent.
\end{proposition}

\begin{proof}
    Balanced strong shift equivalence is generated by one-sided outsplits (which we know preserve the unital dimension data by \cref{lem:one-sided-conjugacy-SE+}) and balanced elementary equivalence, so it suffices to assume that $A$ and $B$ are balanced elementary equivalent. 
    This means that there are rectangular $\N$-matrices $S$, $R_A$, and $R_B$ satisfying
    \[
    A = SR_A,\qquad B = SR_B, \qquad R_A S = R_B S.
    \]
    Observe that $(B,A)$ defines a shift equivalence from $A$ to $B$ of lag $2$ which is clearly unital.
\end{proof}

Since balanced strong shift equivalence is an algebraic characterisation of Matsumoto's one-sided eventual conjugacy, the proposition shows that if one-sided shifts of finite type are Matsumoto eventually conjugate, then the matrices are unitally shift equivalent.
This can also be proven using C*-algebras through Carlsen and Rout's theorem \cite[Theorem 4.1]{Carlsen-Rout} for Matsumoto's eventual conjugacy.

At the level of graphs we have the following corollary.

\begin{corollary}
Let $E$ be a finite graph with no sinks.
If $F$ is either an out-split or a balanced in-split of $E$, then the adjacency matrices of $E$ and $F$ are unitally shift equivalent.
\end{corollary}

The next example shows that the converse of \cref{prop:111->SE+} does not hold.
It is a modification of Kim and Roush's famous counterexample to the Williams conjecture from \cite{Kim-Roush99}.
The details are provided in \cite{Eilers-Ruiz}, and we thank Eilers for allowing us to reproduce the example here.

\begin{example}
Unital shift equivalence does not imply (balanced) strong shift equivalence.
Consider Kim and Roush's counterexamples to the Williams conjecture
\[
A = 
\begin{pmatrix}
    0 & 0 & 1 & 1 & 3 & 0 & 0 \\
    1 & 0 & 0 & 0 & 3 & 0 & 0 \\
    0 & 1 & 0 & 0 & 3 & 0 & 0 \\
    0 & 0 & 1 & 0 & 3 & 0 & 0 \\
    0 & 0 & 0 & 0 & 0 & 0 & 1 \\
    1 & 1 & 1 & 1 & 10 & 0 & 0 \\
    1 & 1 & 1 & 1 & 0 & 1 & 0 
\end{pmatrix} \quad \textup{and} \quad
B = 
\begin{pmatrix}
    0 & 0 & 1 & 1 & 3 & 0 & 0 \\
    0 & 0 & 1 & 1 & 0 & 0 & 0 \\
    0 & 0 & 1 & 1 & 0 & 0 & 0 \\
    0 & 0 & 1 & 1 & 0 & 0 & 0 \\
    0 & 0 & 0 & 0 & 0 & 0 & 1 \\
    4 & 5 & 6 & 3 & 10 & 0 & 0 \\
    4 & 5 & 6 & 3 & 0 & 1 & 0 
\end{pmatrix}.
\]
There are $7\times 7$ matrices $R$ and $S$ that implement a concrete shift equivalence of lag 13 from $A$ to $B$, and the deep result of \cite{Kim-Roush99} is that $A$ and $B$ are \emph{not} strong shift equivalent.
By setting
\[
B' \coloneq 
\begin{pmatrix}
    0 & v \\
    0 & B
\end{pmatrix}, \qquad v \coloneq \underline 1^t RB - \underline 1^t B,
\]
Eilers and Ruiz find matrices $R_1$ and $S_1$ that implement a \emph{unital} shift equivalence from $A$ to $B'$.
Since $B$ and $B'$ are strong shift equivalent, it follows that $A$ and $B'$ are \emph{not} strong shift equivalent, in particular, not balanced strong shift equivalent.
It follows that the converse to \cref{prop:111->SE+} does not hold.
\end{example}

Next, let $A$ be a square $\N$-matrix.
The $i$'th higher power of the one-sided shift of finite type $(\sigma_A, X_A)$ is the system $(\sigma_A^i,X_A)$,
and the higher power $\sigma_A^i$ is one-sided conjugate to the edge shift $(\sigma_{A^i},X_{A^i})$.


\begin{definition}
  Let $A$ and $B$ be square $\N$-matrices wit no zero rows.
  The one-sided edge shifts $(\sigma_A, X_A)$ and $(\sigma_B, X_B)$ have \emph{conjugate higher powers} if all but finitely many powers of $\sigma_A$ and $\sigma_B$ are one-sided conjugate.
\end{definition}

In~\cite{Boyle-Fiebig-Fiebig}, Boyle, D.~Fiebig, and U.~Fiebig study one-sided shifts of finite type that have conjugate higher powers
(they refer to this as one-sided eventual conjugacy).
We present here the main result related to this notion from \cite[Section 8]{Boyle-Fiebig-Fiebig}.
An immediate consequence of this theorem is that having conjugate higher powers is decidable.

\begin{theorem} \label{thm:BFF}
  Let $A$ and $B$ be square $\N$-matrices with no zero rows and let $n \geq \max\{ |A|, |B| \}$.
  Then, $A$ and $B$ determine one-sided shifts of finite type that have conjugate higher powers if and only if 
  the total amalgamations of $A^m$ and $B^m$ agree (up to the same permutation) for $m = n, n+1$.
\end{theorem}

Boyle, Fiebig and Fiebig~\cite[Theorem 8.3]{Boyle-Fiebig-Fiebig} show that the dimension triple is an invariant for having conjugate higher powers.
In fact, their arguments shows that the unit is also preserved, so the dimension \emph{quadruple} is an invariant of having conjugate higher powers.
Specifically, they define an \emph{images group} to any surjective local homeomorphism,
and this coincides with the dimension data (\cite[Theorem 4.5(a)]{Boyle-Fiebig-Fiebig}) also with the unit although this is not mentioned there.
Hence we have the following.

\begin{proposition} 
  Let $A$ and $B$ be square $\N$-matrices.
  If the one-sided shifts $X_A$ and $X_B$ have conjugate higher powers, then $A$ and $B$ are unitally shift equivalent.
\end{proposition}

The example below shows that unital shift equivalence is nontrivial in the sense that it does not coincide with one-sided conjugacy.
The case $k=1$ is from \cite[Example 5.6]{Boyle-Fiebig-Fiebig}, and the details for the general case are in \cite[Example 6.6]{Brix2024}.

\begin{example}\label{ex:BFF}
  For every integer $k\geq 1$, the matrices
  \[
    A_k = 
    \begin{pmatrix}
      2k & 0  & 4k \\
      k  & 2k & 0  \\
      k  & 2k & 0  \\
    \end{pmatrix} \quad \textup{and} \quad
    B_k = (4k)
  \]
  have conjugate higher powers. 
  In particular, they are unitally shift equivalent. 
  The one-sided shifts of finite type that they determine are however not one-sided conjugate as their total amalgamations are distinct.
  The matrices are also balanced strong shift equivalent, see \cite[Example 6.6]{Brix2024}.
\end{example}

Below, we replicate \cite[Example 6.7]{Brix2024} which shows that having conjugate higher powers does not coincide with Matsumoto's eventual conjugacy.

\begin{example} 
  The adjacency matrices of the graphs in~\cite[Example 3.6]{Brix-Carlsen2020} are
  \[
    A = 
    \begin{pmatrix}
      0 & 2 & 2 \\
      1 & 0 & 0 \\
      1 & 0 & 0 
    \end{pmatrix}\quad \textrm{and} \quad
    B = 
    \begin{pmatrix}
      0 & 3 & 1 \\
      1 & 0 & 0 \\
      1 & 0 & 0 
    \end{pmatrix},
  \]
  and the one-sided edge shifts are Matsumoto eventually conjugate, in particular unitally shift equivalent.
  However, the total amalgamations of the powers of $A$ and $B$ do not coincide so they do not have conjugate higher powers. 
\end{example}

We leave the following problems open (see the dashed arrow in the diagram in the introduction).

\begin{question}
  Does having conjugate higher powers imply balanced strong shift equivalence?
  Does it imply strong shift equivalence?
\end{question}

A positive answer to either one of the above questions will resolve the open problem of whether Ashley's matrices below are strong shift equivalent,
since balanced strong shift equivalence implies strong shift equivalence.

\begin{example}\label{ex:Ashley}
Consider Ashley's graph
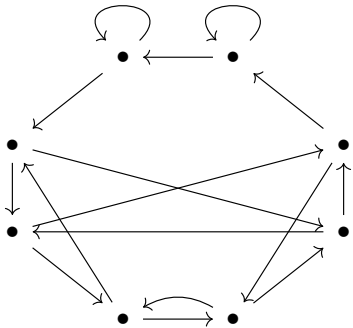
\[
  \begin{tikzcd}
    & \bullet \arrow[dl] \arrow[loop, looseness=5] & \bullet \arrow[l] \arrow[loop, looseness=5]&           \\
    \bullet \arrow[d] \arrow[rrrd] &         &         & \bullet \arrow[ul]  \arrow[ddl] \\
    \bullet \arrow[dr] \arrow[rrru] &         &         & \bullet \arrow[u] \arrow[lll] \\
    & \bullet \arrow[r] \arrow[uul] & \bullet \arrow[ur] \arrow[l, bend right] &           
  \end{tikzcd}
\]
The characteristic polynomial of the adjacency matrix $A$ of this graph is $x^7(x-2)$ and a result of Williams (see e.g.~\cite[Lemma 2.2.6]{Kitchens}) 
implies that the two-sided edge shift is shift equivalent to the matrix $(2)$, the full $2$-shift.
It is still an open problem to determine whether the two systems are in fact two-sided conjugate.
A computation of the higher powers of the adjacency matrix of Ashley’s
graph shows that all but finitely many of the higher powers are one-sided conjugate to the higher powers of the full 2-shift (using~\cref{thm:BFF}), so they have conjugate higher powers.
In particular, they are unitally shift equivalent.

Alternatively, the vectors $\underline 1$ and $\underline 1^t$ are right and left eigenvectors for $A$.  Since the characteristic polyonmial of $A$ is $x^7(x-2)$, we have that the rank of $A^7$ is $1$.  Hence, $A^7 = (2^4 \underline{1})(\underline{1}^t)$ which implies the pair $(2^4 \underline{1}, \underline{1}^t)$ gives a unital shift equivalence from $A$ to $(2)$ of lag $7$.
\end{example}

\begin{example}\label{ex:Rourke}
  Consider Rourke's example of shift equivalent primitive matrices
  \[
    A = 
    \begin{pmatrix}
      1 & 2 & 1 \\
      1 & 1 & 0 \\
      1 & 0 & 1
    \end{pmatrix} \qquad \textrm{and} \qquad
    B = 
    \begin{pmatrix}
      1 & 0 & 1 & 0 & 1 \\
      0 & 1 & 1 & 1 & 0 \\
      1 & 1 & 1 & 0 & 0 \\
      1 & 0 & 0 & 0 & 1 \\
      0 & 1 & 0 & 1 & 0 
    \end{pmatrix}
  \]
  from the errata of \cite{Williams1973}.  A computation of the higher powers shows that the total amalgamation of $B^n$ is $A^n$ for $n\geq 2$,
  so the matrices have conjugate higher powers, in particular they are unitally shift equivalent.

  Effros \cite[p.41]{Effros1979} remarks that Handelman has shown that the matrices are strong shift equivalent, although we are not aware of an explicit proof in the literature.
  The matrices are in fact balanced strong shift equivalent, and we are grateful to Søren Eilers for finding the following set of moves using a computer search.
First consider the matrices
\[
R_0 = \begin{pmatrix}
      1 & 0 & 1 & 0 & 1  \\
      0 & 1 & 0 & 1 & 0  \\
      1 & 1 & 1 & 0 & 0  \\
      1 & 0 & 0 & 0 & 1  \\
      0 & 1 & 0 & 1 & 0  \\
      0 & 0 & 1 & 0 & 0 
    \end{pmatrix} \quad \textup{ and } \quad
    S_0
    \begin{pmatrix}
        1 & 0 & 0 & 0 & 0 & 0 \\
        0 & 1 & 0 & 0 & 0 & 1 \\
        0 & 0 & 1 & 0 & 0 & 0 \\
        0 & 0 & 0 & 1 & 0 & 0 \\
        0 & 0 & 0 & 0 & 1 & 0 \\
    \end{pmatrix}
\]
and observe that
\[
B = S_0 R_0 \quad \textup{ and }\quad  
B' \coloneq \begin{pmatrix}
      1 & 0 & 1 & 0 & 1 & 0 \\
      0 & 1 & 0 & 1 & 0 & 1 \\
      1 & 1 & 1 & 0 & 0 & 1 \\
      1 & 0 & 0 & 0 & 1 & 0 \\
      0 & 1 & 0 & 1 & 0 & 1 \\
      0 & 0 & 1 & 0 & 0 & 0
    \end{pmatrix} = R_0 S_1.
\]
Next, consider
\[
R_1' \coloneq \begin{pmatrix}
      1 & 0 & 1 & 0 & 1 & 0 \\
      0 & 1 & 0 & 1 & 0 & 1 \\
      1 & 1 & 1 & 0 & 0 & 1 \\
      1 & 0 & 0 & 0 & 1 & 0 \\
      0 & 0 & 1 & 0 & 0 & 0
\end{pmatrix}, \quad
R_1'' \coloneq \begin{pmatrix}
      1 & 0 & 1 & 0 & 1 & 0 \\
      0 & 1 & 0 & 1 & 0 & 1 \\
      1 & 0 & 1 & 0 & 1 & 1 \\
      1 & 0 & 0 & 0 & 1 & 0 \\
      0 & 0 & 1 & 0 & 0 & 0
\end{pmatrix}, \quad \textup{ and } \quad
S_1' \coloneq \begin{pmatrix}
    1 & 0 & 0 & 0 & 0 \\
    0 & 1 & 0 & 0 & 0 \\
    0 & 0 & 1 & 0 & 0 \\
    0 & 0 & 0 & 1 & 0 \\
    0 & 1 & 0 & 0 & 0 \\
    0 & 0 & 0 & 0 & 1 \\
\end{pmatrix},
\]
and observe that 
\[
B' = S_1' R_1', \quad 
B'' \coloneq \begin{pmatrix}
      1 & 0 & 1 & 0 & 1 & 0 \\
      0 & 1 & 0 & 1 & 0 & 1 \\
      1 & 0 & 1 & 0 & 1 & 1 \\
      1 & 0 & 0 & 0 & 1 & 0 \\
      0 & 1 & 0 & 1 & 0 & 1 \\
      0 & 0 & 1 & 0 & 0 & 0
    \end{pmatrix} 
    = S_1' R_1'', \quad \textup{ and } \quad R_1'S_1' = R_1'' S_1'.
\]
This is a balanced insplit between $B'$ and $B''$ which is move $\texttt{\textup{(I+)}}$ \text in \cite[Chapter 2.2]{Eilers-Ruiz}.
The matrix $A$ is now the total amalgamation of $B''$.
We conclude that $A$ and $B$ are balanced strong shift equivalent (in particular, strong shift equivalent).
\end{example}

\section{Continuous orbit equivalence} \label{sec:SE+->COE}
In this section, we prove that unitally shift equivalent matrices determine one-sided shifts of finite type that are continuously orbit equivalent.
Continuous orbit equivalence is an important link between symbolic dynamics and C*-algebras as well as topological full groups \cite{Matsumoto2010,Matui2015}.
A consequence of our result is that unitally shift equivalent matrices will have isomorphic topological full groups.
The theorem is a one-sided analogue of a recent result of Boyle showing that shift equivalence implies flow equivalence for all (not necessarily irreducible) shifts of finite type. 
From a shift equivalence, Boyle constructs an explicit $\mathsf{SL}$ equivalence thus proving flow equivalence.
This requires a careful study of the poset of irreducible components of the matrix.

Let $A$ and $B$ be square $\N$-matrices with no zero rows.
Recall from \cite[Section 5]{Matsumoto2010} (see also \cite{BCW17}) that the edge shifts $(\sigma_A, X_A)$ and $(\sigma_B, X_B)$ are \emph{continuously orbit equivalent} if there are continuous functions $k_A,l_A\colon X_A \to \N$ and $k_B,l_B\colon X_B \to \N$ and a homeomorphism $h\colon X_A \to X_B$ satisfying 
\begin{align}
    \sigma_B^{l_A(x)}\circ h(x) &= \sigma_B^{k_A(x)}\circ h(\sigma_A(x)), \\
    \sigma_A^{l_B(y)}\circ h^{-1}(y) &= \sigma_A^{k_B(y)}\circ h^{-1}(\sigma_B(y)), 
\end{align}
for all $x\in X_A$ and $y\in X_B$.

A point $x\in X_A$ is \emph{eventually periodic} if there are distinct $p,q\in \N$ such that $\sigma_A^p(x) = \sigma_A^q(x)$.
A continuous orbit equivalence $h\colon X_A \to X_B$ preserves eventually periodic points if $h(x)$ is eventually periodic precisely if $x$ is eventually periodic.
This extra condition of preserving eventually periodic points is important for connections to \'etale groupoids and topological full groups.


Boyle used $\mathsf{SL}$ equivalence of matrices to characterise flow equivalence, and there is a unital version called $\mathsf{SL_+}$ equivalence that we introduce below.
Although we know that $\mathsf{SL_+}$ equivalence implies continuous orbit equivalence, we do not know if the converse holds. 
We take the opportunity to reiterate a related question, see \cite[Conjecture~6.2.1]{Eilers-Ruiz}
(see \cite[Definition 2.3.1]{Eilers-Ruiz} for the definition of unital reduction):

\begin{question}
Given square $\N$-matrices $A$ and $B$, if $(\sigma_A, X_A)$ and $(\sigma_B, X_B)$ are continuously orbit equivalent, is it possible to transform $A$ into $B$ using out-splits, balanced in-splits, and unital reduction (and their inverses)?
\end{question}


In order to deal with not necessarily irreducible shifts of finite type it is important to keep track of the structure of its irreducible components.
A square $\N$-matrix $A$ admits a canonical partially ordered set (poset) of its irreducible components $\cP^A$ where for $p,q\in \cP^A$, we have $p \leq q$ when there exist an integer $k\geq 1$ and indices $i$ and $j$ for components $p$ and $q$, respectively, such that $A^k(i,j) \neq 0$.
Given a poset $\cP$, a matrix $A$ is $\cP$-partitioned if $\cP$ defines partitions of the row and column sets of $A$ in a way that respects its irreducible components.
We refer the reader to (and assume familiarity with) \cite[Section 4]{Boyle2024} for a careful development of $\cP$-partitioned matrices.

Below we state and prove the main result of the section and the paper.
In order to verify that unital shift equivalence implies continuous orbit equivalence (which is the one-sided analogue of flow equivalence), we go via $\mathsf{SL_+}$ equivalence \cite[Section 5]{Arklint-Eilers-Ruiz2022}.
This is a unital version of $\mathsf{SL}$ equivalence, which is the key notion that Boyle uses to characterise flow equivalence.
Roughly, two matrices $(A,B)$ are $\mathsf{SL}$ equivalent if there are partitioned square integral matrices $U$ and $V$ for which every diagonal block has determinant 1 satisfying $U A V = B$ 
in a way that respects the poset structures, see \cite[Definition 5.3]{Boyle2024} for the precise definition (see also \cite[Section 2]{Boyle-Huang} where this first appeared).
The matrix $U$ induces a map on the cokernels $\check{U}\colon \Z^{|A|}/A \Z^{|A|} \to \Z^{|B|}/ B\Z^{|B|}$ via multiplication by $U$.
Let $C$ and $D$ be matrices, and let $v$ and $w$ be row vectors.   The pairs 
\[ (C, v) \quad \text{and } \quad (D, w) \]
are $\mathsf{SL}_+$ equivalent if there is an $\mathsf{SL}$ equivalence $(U,V)$ from $C$ to $D$ such that 
\[
\check{U}( \underline{1} + v^t + C \Z^{|C|}) = \underline 1 +w^t+ D\Z^{|D|}.
\]

We say a block upper triangular matrix $A$ with a given poset structure is in \emph{canonical form} provided that
\begin{enumerate}
    \item every diagonal entry is positive;
    \item If $A^n(i,j)>0$ for some $n$, then $A(i,j)>0$; and 
    \item each irreducible diagonal block is either the 1 by 1 matrix $(1)$ or a matrix with dimension at least 3, with smith normal form having at least two ones in its diagonal, and with all diagonal entries having values at least 2.
\end{enumerate} 
Our definition of canonical form of a matrix is chosen in such a way that for a nonzero row vector $v_A$ of non-negative integer entries, the graph with adjacency matrix \[ \begin{pmatrix} 0 & v_A \\ 0 & A \end{pmatrix} \] is in augmented canonical form in the sense of \cite[Definitions~4.2]{Arklint-Eilers-Ruiz2022}.

A pair $(A,B)$ of canonical form matrices with a given poset structure is in \emph{standard form} provided that the corresponding diagonal blocks of $A$ and $B$ have the same dimensions.  If $(A,B)$ is in standard form, then for any nonzero row vectors $v_A$ and $v_B$ of non-negative integer entries, the graphs with adjacency matrices 
\[ \begin{pmatrix} 0 & v_A \\ 0 & A \end{pmatrix} \quad \text{and} \quad \begin{pmatrix} 0 & v_B \\ 0 & B \end{pmatrix}\]
are in augmented standard form in the sense of \cite[Definition~5.1]{Arklint-Eilers-Ruiz2022}.  If $(A,B)$ are in standard form and $(I-A^t, v_A)$ and $(1-B^t, v_B)$ are $\mathsf{SL}_+$-equivalent, then by \cite[Theorem~5.4]{Arklint-Eilers-Ruiz2022}, the edge shifts $(\sigma_{C} ,X_C)$ and $(\sigma_D, X_D)$ are continuously orbit equivalent, where
\[C = \begin{pmatrix} 0 & v_A \\ 0 & A \end{pmatrix} \quad \text{and}  \quad D= \begin{pmatrix} 0 & v_B \\ 0 &B\end{pmatrix}.\]

Interestingly, since stabilisations do not work as well with continuous orbit equivalence, we need a result of Boyle and Huang, and since continuous orbit equivalence is not as well understood as flow equivalence, we also need to invoke C*-algebra theory to reach the conclusion that unitally shift equivalent matrices implies the edge shifts are continuously orbit equivalent.

\begin{theorem}\label{thm:SE+->COE}
Let $A$ and $B$ be square $\N$-matrices with no zero rows.
If $A$ and $B$ are unitally shift equivalent, then the edge shifts $(\sigma_A,X_A)$ and $(\sigma_B, X_B)$ are continuously orbit equivalent in a way that preserves eventually periodic points. 
\end{theorem}

\begin{proof}
By \cite[Proposition 5.2.1]{Eilers-Ruiz}, we may assume that the matrices are of the form
\[
A = 
\begin{pmatrix}
0 & v_A \\
0 & A'
\end{pmatrix}\quad \textup{and}\quad 
B = \begin{pmatrix}
0 & v_B \\
0 & B'
\end{pmatrix}
\]
where $v_A$ and $v_B$ are row vectors, and $A'$ and $B'$ are essential matrices (i.e. there are no zero rows or columns) that are block upper triangular with irreducible diagonal blocks and the greatest common divisor of each diagonal block is 1.  If $A$ is already an essential matrix that is block upper triangular with irreducible diagonal blocks and the greatest common divisor of each diagonal block is 1, then $A'=A$ and there is no $v_A$.  Similarly for $B$.  Moreover, we may assume the corresponding diagonal blocks of $A'$ and $B'$ are of the same size (here we are using the fact that an essential matrix that is shift equivalent to a permutation matrix is a permutation matrix of the same size).
This process uses out-splits and balanced in-splits both of which preserve the Matsumoto eventual conjugacy class (and hence the unital shift equivalence class) of the matrices
and the continuous orbit equivalence class of the one-sided edge shifts, see \cref{prop:111->SE+}.

Suppose $A$ and $B$ are unitally shift equivalent.
The matrices $\begin{pmatrix}
    v_A \\ A'
\end{pmatrix}$ and $\begin{pmatrix}
    0 & I
\end{pmatrix}$
define a lag $1$ shift equivalence from $A$ to $A'$ (an elementary strong shift equivalence),
and this induces an isomorphism of dimension triples $(G_A,G_A^+,\theta_A) \cong (G_{A'},G_{A'}^+,\theta_{A'})$.
We let $w_A$ be the image of $u_A$ under this isomorphism.
Similarly, there is a shift equivalence from $B$ to $B'$, and we define $w_B$ analogously to obtain isomorphisms
\begin{align}
    (G_A,G_A^+,\theta_A,u_A) &\cong (G_{A'},G_{A'}^+,\theta_{A'},w_A) \\
    (G_B,G_B^+,\theta_B,u_B) &\cong (G_{B'},G_{B'}^+,\theta_{B'},w_B).
\end{align}
As we assume that $A$ and $B$ are unitally shift equivalent, the left most quadruples are isomorphic, so the right most quadruples are isomorphic.
Using a similar argument as in \cref{prop:SE+-dimension-data}, we may choose a shift equivalence $(R,S)$ from $A'$ to $B'$ satisfying $\hat{R}w_A = (\hat{B'})^k w_B$, for some $k\in \N$.  Note that $[w_A] = [\underline{1}+ v_A^t]$ in $\BF(A')$ and $[w_B] = [\underline{1}+ v_B^t]$ in $\BF(B')$

Next, we apply Boyle's observation \cite[Proposition 4.3 and Theorem 4.7]{Boyle2024} to see that we have a partitioned polynomial shift equivalence from $A'$ to $B'$.  The partitioned polynomial shift equivalence equation (with the variable set to $1$) takes the form

\begin{equation} \label{eq:Boyle-PSE}
  \begin{pmatrix}
    W^t & -S^t \\
    R^t & I-(B')^t
  \end{pmatrix}
  \begin{pmatrix}
    I-(A')^t & 0 \\
    0 & I
  \end{pmatrix}
  \begin{pmatrix}
    I & S^t \\
    -R^t & I-(B')^t
  \end{pmatrix} 
  =
\begin{pmatrix}
  I & 0 \\
  0 & I - (B')^t
\end{pmatrix}
\end{equation}

 where $W = I + A' + \cdots + (A')^{\ell - 1}$.
 This can be verified directly using the shift equivalence relations, and it defines an $\mathsf{SL}$ equivalence between 
 $\begin{pmatrix} I-(A')^t & 0 \\ 0 & I \end{pmatrix}$ and $\begin{pmatrix} I & 0 \\ 0 & I-(B')^t \end{pmatrix}$ (see \cite[Proof of Theorem 1.1]{Boyle2024}).
 
 Observe that $U'\coloneq \begin{pmatrix} W^t & -S^t \\ R^t & I-(B')^t \end{pmatrix}$ and $R^t$ induce the same maps on the Bowen--Franks groups.
 Indeed, we observe that 
 \[
 \begin{pmatrix}
    W^t & -S^t \\
    R^t & I-(B')^t
  \end{pmatrix} 
  \begin{pmatrix}
      x \\ y
  \end{pmatrix}
  =
  \begin{pmatrix}
      0 \\ R^t x
  \end{pmatrix}
  +
  \begin{pmatrix}
  I & 0 \\
  0 & I - (B')^t
\end{pmatrix}
\begin{pmatrix}
    x-S^t y \\ y
\end{pmatrix}
 \]
 and, on the other hand, 
 \[
 \begin{pmatrix}
     0 & I \\ 
     R^t & 0
 \end{pmatrix}
 \begin{pmatrix}
      x \\ y
  \end{pmatrix}
  =
    \begin{pmatrix}
      0 \\ R^t x
  \end{pmatrix}
  +
  \begin{pmatrix}
  I & 0 \\
  0 & I - (B')^t
\end{pmatrix}
\begin{pmatrix}
    y \\ 0
\end{pmatrix}.
 \]

 By our choice of $(R,S)$, it follows from the observation \eqref{eq:R-hat-check} that $U'$ induces a map on the Bowen--Franks groups.  Since $U'$ and $V'$ are $\mathsf{SL}$ matrices, the pair $(U', V')$ induces an $\mathsf{SL}$ isomorphism between the K-webs (as defined \cite[Section~3]{Boyle-Huang}) of $I-(A')^t$ and $I-(B')^t$.  By \cite[Corollary~4.8]{Boyle-Huang}, this K-web isomorphism is induced by a possibly different $\mathsf{SL}$ equivalence $(U,V)$ between $I-(A')^t$ and $I-(B')^t$.
 Since $R^t$, $U'$, and $U$ all induce the same isomorphism of Bowen--Franks groups, $U$ must send the class of $w_A$ in $\BF(A')$ to the class of $w_B$ in $\BF(B')$. Hence, we conclude that $(U,V)$ is an $\mathsf{SL_+}$ equivalence between $(I-(A')^t, v_A)$ and $(I-(B')^t, v_B)$.

One can check that the algorithms described in \cite[Proposition 4.3 and Lemma 5.2]{Arklint-Eilers-Ruiz2022} to put a pair of matrices into augmented standard form preserve the Matsumoto eventual conjugacy class as well as induce $\mathsf{SL}$ isomorphisms between the K-webs.  By \cite[Corollary~4.8]{Boyle-Huang}, we may now further assume that $( (I-(A')^t , I- (B')^t)$ are in standard form and there is an $\mathsf{SL}_+$ equivalence between $(I-(A')^t, v_A)$ and $(I-(B')^t, v_B)$.  It now follows from \cite[Theorem 5.4]{Arklint-Eilers-Ruiz2022} (see also \cite[Theorem 4.2.2]{Eilers-Ruiz}) that the one-sided edge shifts are continuously orbit equivalent.
This uses the fact that a diagonal-preserving *-isomorphism of the associated Cuntz--Krieger C*-algebras implies that the shifts are continuously orbit equivalent in a way that preserves eventually periodic points, see \cite[Theorem 5.1]{BCW17} and \cite[Theorem 5.3]{Arklint-Eilers-Ruiz2018}.
\end{proof}

As in the proof above, if $A$ and $B$ are unitally shift equivalent, then their \'etale groupoids are isomorphic (see \cite[Theorem 5.1]{BCW17} and \cite[Theorem 5.3]{Arklint-Eilers-Ruiz2018}), so their topological full groups are also isomorphic, see \cite{Matui2015}. 
We record this as a separate corollary.

\begin{corollary}
    Let $A$ and $B$ be square $\N$-matrices with no zero rows.
    If $A$ and $B$ are unitally shift equivalent, then the topological full groups of $(\sigma_A,X_A)$ and $(\sigma_B, X_B)$ are isomorphic.
\end{corollary}

The next corollary relates to Hazrat's graded classification conjecture since the unital dimension data of a matrix $A$ coincides with the unital graded K-theory of the Leavitt path algebra attached to $A$.

\begin{corollary}\label{cor:LPA-isomorphism}
    Let $A$ and $B$ be square $\N$-matrices with no zero rows.
    If $A$ and $B$ are unitally shift equivalent, then the Leavitt path algebras of $A$ and $B$ are isomorphic.
\end{corollary}

\end{document}